\documentclass[11pt]{article}
\usepackage{amsmath}
\usepackage{amssymb}
\usepackage{amsthm}
\usepackage{lineno}
\usepackage[mathscr]{eucal}
\theoremstyle{definition}
\newtheorem{definition}{Definition}
\newtheorem{theorem}[definition]{Theorem}
\newtheorem{proposition}[definition]{Proposition}
\newtheorem{lemma}[definition]{Lemma}

\theoremstyle{remark}
\newtheorem{remark}[definition]{Remark}
\newtheorem{example}[definition]{Example}
\newcounter{enumctr}

\newcommand{\R}{\mathbb{R}}

\newcommand{\C}{\mathbb{C}}
\newcommand{\id}{\hbox{id}}

\renewcommand{\phi}{\varphi}

\providecommand{\keywords}[1]{\textbf{\textbf{Key words: }} #1}
\begin{document}
%
%
%
\title{\vspace*{-10mm}
On asymptotic properties of solutions to fractional differential equations}
\author{
N.D.~Cong\footnote{\tt ndcong@math.ac.vn, \rm Institute of Mathematics, Vietnam Academy of Science and Technology, 18 Hoang Quoc Viet, 10307 Ha Noi, Viet Nam.}, \;
%
%
%
H.T.~Tuan\footnote{\tt httuan@math.ac.vn, \rm Institute of Mathematics, Vietnam Academy of Science and Technology, 18 Hoang Quoc Viet, 10307 Ha Noi, Viet Nam.}
\;\;and\;\;
Hieu Trinh\footnote{\tt  hieu.trinh@deakin.edu.au, \rm School of Engineering, Faculty of Science Engineering and Built Environment, Deakin University, Geelong, VIC 3217, Australia}}
\maketitle
\begin{abstract}
We present some distinct asymptotic properties of solutions to Caputo fractional differential equations (FDEs). First, we show that the non-trivial solutions to a FDE can not converge to the fixed points faster than $t^{-\alpha}$, where $\alpha$ is the order of the FDE. Then, we introduce the notion of Mittag-Leffler stability which is suitable for systems of fractional-order. Next, we use this notion to describe the asymptotical behavior of solutions to FDEs by two approaches: Lyapunov's first method and Lyapunov's second method. Finally, we give a discussion on the relation between Lipschitz condition, stability and speed of decay, separation of trajectories to scalar FDEs.
\end{abstract}
%
%

\section{Introduction}
Fractional derivatives provide an excellent instrument for the description of memory and hereditary properties of various materials and  processes. This is the main advantage of fractional derivatives in comparison  with classical integer-order models, in which such effects are in fact  neglected \cite{Podlubny}. Therefore, during  the  last  decades, fractional  calculus  has  been  applied  to  almost  every  field  of  science, engineering, and mathematics.  The areas where fractional calculus  has made a profound impact include viscoelasticity and rheology, electrical  engineering,   electrochemistry,   biology,   biophysics  and  bioengineering, signal and image processing, mechanics, mechatronics, physics, and control  theory. For more details we refer the reader the monographs \cite{OS_74, Podlubny, BK_15, Kai, KST_06, LLD_09, SKM} and the references therein.  The mathematical modeling and simulation of systems and  processes, based on the description of their properties in terms of fractional derivatives, naturally leads to differential equations of fractional order and to the necessity to solve such equations. However, most of the fractional differential equations used to describe practical problems can not be solved explicitly. 

Many important problems  of the qualitative theory to fractional-order systems deal with asymptotic properties of solutions what happens after a long period of time. If a particular solution is well understood, it is natural to ask next whether a small change in the initial condition will lead to similar behavior. Stability theory addresses the following questions: how do the trajectories of solutions change under small perturbations? will the solutions starting near to a given equilibrium point converge to that equilibrium point, and, if yes, with what rate of convergence?

In his seminal 1892 thesis [L92] Lyapunov proposed two main methods  for investigation asymptotic properties of solution of ODEs as  follows.\smallskip

\noindent $\bullet$ Lyapunov's first method (reduction method):  the key feature of this method is that one reduces the  original problem to a much simpler one---linearization of the nonlinear equation near an equilibrium point. Then the stability of the resulting linearized equation can be solved and used for deducting the asymptotic properties of the original equation.\smallskip

\noindent $\bullet$ Lyapunov's second method (direct method): use the action of the system on a specific function (called Lyapunov function) to deduct the asymptotic properties of the system without the need to solve the system explicitly.\smallskip

The two Lyapunov's methods proved to be powerful tools in the classical theory of ordinary differential equations, which help the research and understanding the behavior of solutions of a system of ordinary differential equations (ODEs). It is natural to expect that the Lyapunov's methods may work for FDEs as well since the fractional order systems are  generalizations of integer-order systems. However, one should take care of many distinct features of "purely" fractional-order systems, especially the nonlocal property and the long memory of the system.

The first work concerning with Lyapunov's first method for fractional-order systems is the paper \cite{A_07}. Using linearization, the authors proposed a criterion to test the stability of a fractional-order predator-prey model and a fractional-order rabies model. They also examined these results by a numerical example. However, no rigorous  mathematical proof is given in that paper.  After that in \cite{WWL,CCWY_12,ZTC_15}, the authors formulated theorems on linearized stability. Unfortunately,  as showed in \cite[Remark 3.7]{CST_16}, these papers contain some serious flaws in the proof of the linearization theorems. Using other tools, the authors of \cite{CST_16} improved the assertions presented in the papers mentioned above and gave a powerful stability criterion. However,  the convergence rate of solutions to an equilibrium point is still unavailable.

Although several results of the Lyapunov's second method for fractional-order nonlinear systems have been published. We list here some typical contributions \cite{A_17,CDSZ_17,CMC_14,DQW_14,Li_1,Li_2, MCG_15,YC_16,Zhou}. However, the development of this theory is still in its infancy and requires further investigation. One of the reasons for this might be that computation and estimation of fractional derivatives of Lyapunov candidate functions are very complicated due to the fact that the well-known Leibniz rule does not hold true for such derivatives. On the other hand, in contrast to classical derivative, there is no an acceptable geometrical or physical interpretation of fractional derivatives. To the best of our knowledge, the common strategy in study the stability of FDEs by Lyapunov's second method as follows. The authors combined effective fractional derivative inequalities  \cite[inequalities (6) and (16)]{CMC_14}, \cite[inequality (24)]{MCG_15}, \cite[inequality (10)]{CDSZ_17} and mains results in \cite{Li_1,Li_2} to obtain the estimation of solutions to FDEs. However,  in those papers there are some shortcomings of that approach and some flaws in the proofs, which were shown in \cite{Tuan_Hieu}.
Recently, using other tools the authors of  \cite{Tuan_Hieu} were able to avoid the shortcomings and flaws mentioned above and proposed a rigorous method of fractional Lyapunov candidate functions to study the weakly asymptotical stability for FDEs.

In this paper, we will develop a framework to study the asymptotical behavior of solution to FDEs by two directions: Lyapunov's first method and Lyapunov's second method. We will improve the existing results and solve the open problems in the literature.

The rest of the paper is organized as follows. In Section \ref{pre}, some important notions and elementary results concerning with fractional calculus and FDEs are recalled. In Section \ref{sec.power.rate}, we first show that every nontrivial solution of a FDE having Lipschitz continuous "vector field" does not converge to the equilibrium point of the FDE with a rate faster than $t^{-\alpha}$ where $\alpha$ is the order of the equation. Then, based on the role of Mittag-leffler function in expression of solution to linear FDEs and its asymptotical behavior, we introduce the notion of Mittag-Leffler stability to characterize the decay rate of solution to FDEs around fixed points of the "vector field". The suitability and usefulness of this definition will be specified in the next sections. In Section \ref{first_method}, we develop a Lyapunov's first method for a FDE linearized around its equilibrium points. Our strategy is to combine  a variation of constant formula, properties of Mittag-Leffler function, Lyapunov--Perron approach and a new weighted norm  to obtain the Mittag-Leffler stability of fixed points. Using comparison principles, a characterization of functions having fractional derivative and an inequality
concerning with fractional derivative of convex functions, in Section \ref{sec.second-Lyapunov} we develop a Lyapunov's second method for FDEs. Some examples are also presented to illustrate for the theoretical result. Finally, in Section \ref{final_sec} we discuss relations between Lipschitz condition, stability and speed of decay, separation of trajectories to FDEs. In particular, we give an example showing that Mittag-Leffler stability is strictly stronger than asymptotic stability, and another example showing that without Lipschitz condition we may encounter the non-uniqueness of solutions to FDE and that alone may lead to instability of the equilibrium point although almost all solutions tend to the equilibrium point with a power rate. At the end of this section, we prove a distinct property of solution to scalar FDEs in comparison to solution of general higher dimensional FDEs: two trajectories starting from two different initial conditions do not intersect.

To conclude this part we introduce notations which are used through the paper. Denote by $\R$, $\R_{\geq 0}$ and $\C$ the set of real numbers, non-negative numbers and complex numbers, respectively. For some arbitrary positive constant integer $d$, let $\R^d$ and $\C^d$ be the $d$-dimensional Euclidean spaces with the scalar product $\langle \cdot,\cdot \rangle$ and the norm $\|\cdot\|$. For a Banach space $X$ with the norm $\|\cdot\|$, $x\in X$ and $r>0$, let $B_X(x,r)$ be the closed ball with the center at $x$ and the radius $r>0$. For some $T>0$, denote by $C([0,T],X)$ the linear space of continuous functions $\varphi:[0,T]\rightarrow X$ and by $C_\infty([0,T],X)$ the  normed space of $C([0,T],X)$ equipped with the norm
\[
\|\varphi\|_\infty:=\sup_{t\in [0,T]}\|\varphi(t)\|<\infty
\] 
for any $\varphi\in C_\infty([0,T],X)$. It is obvious that $(C_\infty([0,T],X),\|\cdot\|_\infty)$ is a Banach space. Finally, for $\alpha \in (0,1]$ we mean $\mathcal{H}^\alpha([0, T],\R^d)$ the standard H\"older space consisting of
functions $v\in C([0, T],\R^d)$ such that
\[
\|v\|_{\mathcal{H}^\alpha}:=\max_{0\leq t\leq T}\|v(t)\|+\sup_{0\leq s<t\leq T}\frac{\|v(t)-v(s)\|}{(t-s)^\alpha}<\infty
\]
and by $\mathcal{H}^\alpha_0([0,T],\R^d)$ the closed subspace of $\mathcal{H}^\alpha([0, T],\R^d)$ consisting of functions $v \in \mathcal{H}^\alpha([0, T],\R^d)$ such that
\[
\sup_{0\leq s<t\leq T,t-s\leq \varepsilon}\frac{\|v(t)-v(s)\|}{(t-s)^\alpha}\to 0\quad \text{as}\;\varepsilon\to 0.
\]
\section{Preliminaries}\label{pre}
We recall briefly important notions of fractional calculus and some fundamental results concerning with fractional differential equations.

Let $\alpha\in (0,1)$, $[0,T]\subset \R$ and $x:[0,T]\rightarrow \R$ satisfy $\int_0^T|x(\tau)|\;d\tau<\infty$. Then, the \emph{Riemann--Liouville integral of order $\alpha$} is defined by
\[
I_{0+}^{\alpha}x(t):=\frac{1}{\Gamma(\alpha)}\int_0^t(t-\tau)^{\alpha-1}x(\tau)\;d\tau\quad \hbox{ for } t\in (0,T],
\]
where the Gamma function $\Gamma:(0,\infty)\rightarrow \R$ is defined as
\[
\Gamma(\alpha):=\int_0^\infty \tau^{\alpha-1}\exp(-\tau)\;d\tau,
\]
see e.g., Diethelm \cite{Kai}. The \emph{Riemann--Liouville derivative of fractional-order $\alpha$} is given by
\[
^{R}D_{0+}^\alpha x(t):=(D I_{0+}^{1-\alpha}x)(t) \quad\forall t\in (0,T],
\]
where $D=\frac{d}{dt}$ is the classical derivative. In the case the Riemann--Liouville derivative of $x(\cdot)$ exists, the \emph{Caputo fractional derivative} $^{C\!}D_{0+}^\alpha x$ of this function is defined by
\[
^{C\!}D_{0+}^\alpha x(t):=^{R\!}D_{0+}^\alpha(x(t)-x(0)),\qquad \hbox{ for } t\in (0,T],
\]
see \cite[Definition 3.2, pp. 50]{Kai}. The Caputo fractional derivative of a $d$-dimensional vector function $x(t)=(x_1(t),\dots,x_d(t))^{T}$ is defined component-wise as $$^{C\!}D^\alpha_{0+}x(t)=(^{C\!}D^\alpha_{0+}x_1(t),\dots,^{C\!}D^\alpha_{0+}x_d(t))^{T}.$$

Denote by $I^\alpha_{0+}C([0,T],\R^d)$ the space of functions $\varphi:[0,T]\rightarrow \R^d$ such that there exists a function $\psi\in C([0,T],\R^d)$ satisfying $\varphi=I^\alpha_{0+}\psi$. Due to \cite[Theorem 5.2, pp. 475]{Vainikko}, we have the following characterization of functions having Caputo fractional derivative.
\begin{theorem}\label{Theorem1}
For $\alpha\in (0,1)$ and a function $v\in C([0,T],\R^d)$, the following conditions (i), (ii), (iii)  are equivalent:
\begin{itemize}
\item[(i)] the fractional derivative $^{C}D^\alpha_{0+}v\in C([0,T],\R^d)$ exists;
\item[(ii)] a finite limit $\lim_{t\to 0}\frac{v(t)-v(0)}{t^\alpha}:=\gamma$ exists, and
\[\sup_{0<t\leq T}\left\|\int_{\theta t}^t\frac{v(t)-v(\tau)}{(t-\tau)^{\alpha+1}}\;d\tau\right\|\to 0\quad \text{as}\; \theta\to 1;\]
\item[(iii)] $v$ has the structure $v-v(0)=t^\alpha \gamma+v_0$, where $\gamma$ is a constant vector, $v_0\in \mathcal{H}_0^\alpha ([0,T],\R^d)$, and $\int_0^t (t-\tau)^{-\alpha-1} (v(t)-v(\tau))d\tau=:w(t)$ converges for every $t\in (0,T]$ defining a function $w\in C((0,T],\R^d)$ which has a finite limit $\lim_{t\to 0}w(t)=:w(0)$.
\end{itemize}
For $v\in C([0,T],\R^d)$ having fractional derivative $^{C}D^\alpha_{0+}v\in C([0,T],\R^d)$, it holds $^{C}D^\alpha_{0+}v(0)=\Gamma(\alpha+1)\gamma$, and
\begin{align*}
^{C}D^\alpha_{0+}v(t)&=\frac{v(t)-v(0)}{\Gamma(1-\alpha)t^\alpha}\\
&\hspace{1cm}+\frac{\alpha}{\Gamma(1-\alpha)}\int_0^t \frac{v(t)-v(\tau)}{(t-\tau)^{\alpha+1}}d\tau,\quad 0<t\leq T.
\end{align*}
\end{theorem}
Let $x_0\in \R^d$, $K>0$, $G:=\{(t,x):0\leq t\leq T,\; \|y-x_0\|\leq K\}$ and $f:G\rightarrow \R^d$ is a continuous. Consider a initial value problem of order $\alpha$ in the form
\begin{align}
\label{p_1}^{C}D^\alpha_{0+}x(t)&=f(t,x(t)),\quad t>0,\\
\label{p_2}x(0)&=x_0
\end{align}
Using Theorem \ref{Theorem1} and the arguments as in \cite[Lemma 6.2, p. 86]{Kai} we obtain the following result.
\begin{lemma}\label{equi_l}
 A function $y\in B_{C([0,T],\R^d)}(x_0,K)$ is a solution of  the problem \eqref{p_1}-\eqref{p_2} if and only if it satisfies the Volterra integral equation
\[
y(t)=x_0+\frac{1}{\Gamma(\alpha)}\int_0^t (t-\tau)^{\alpha-1}f(\tau,y(\tau))\;d\tau,\quad t\in [0,T].
\]
\end{lemma}
Using this lemma one derive the following results on existence and uniqueness of solution to the problem \eqref{p_1}-\eqref{p_2}.
\begin{theorem}[Local existence]\label{exist_result_1}
There exists $T_b(x_0)\in (0,T)$ such that the problem \eqref{p_1}-\eqref{p_2} has a solution $\varphi(\cdot,x_0)\in C([0,T_b(x_0)],\R^d)$. Moreover, for any $0\leq t\leq T_b(x_0)$ we have $(t,\varphi(t,x_0))\in G$.
\end{theorem}
\begin{proof}
Using the same arguments as in the proof of \cite[Theorem 6.1, p. 86]{Kai}.
\end{proof}
\begin{theorem}[Existence of unique solution on maximal interval of existence]\label{exist_result_2}
Assume additionally that the function $f(\cdot,\cdot)$ is uniformly Lipschitz continuous with respect to the second variable on $G$. 
 Then there exists (a maximal time) $T_b(x_0)\in (0,T]$ such that the problem \eqref{p_1}-\eqref{p_2} has a unique solution $\varphi(\cdot,x_0)\in C([0,T_b(x_0)],\R^d)$. Moreover, for any $0\leq t\leq T_b(x_0)$ we have $(t,\varphi(t,x_0))\in G$, and $(T_b(x_0),\varphi(T_b(x_0),x_0))\in \partial G$, i.e. either $T_b(x_0) = T$, or $T_b(x_0)<T$ and $\|\varphi(T_b(x_0),x_0)-x_0\|=K$.
\end{theorem}
\begin{proof}
The proof is followed directly from \cite[Proposition 4.6, p.\ 2892]{LeiLi}.
\end{proof}
\section{Asymptotic behavior of solutions to FDEs}\label{sec.power.rate}
In this section we study asymptotic properties of solutions to fractional differential equations and show some distinct features compared to that of solutions to ordinary differential equations. We first show that a solution of fractional differential equations does not converge to an equilibrium point with exponential rate. Then, we present various notions of stability of solutions to FDEs, some of them are completely analogous to that of the ODEs, but one, namely the Mittag-Leffler stability is a new notion of stability which is suitable for systems of fractional-order.
\subsection{Solution of FDEs cannot decay faster than power rate}
Consider a nonlinear fractional system of order $\alpha\in (0,1)$ in the form
\begin{equation}\label{introEq}
^{C}D_{0+}^\alpha x(t)=g(t,x(t)),\quad t>0,
\end{equation}
where $g:\R_{\geq 0}\times \R^d\rightarrow \R^d$ satisfies the three conditions: 
\begin{itemize}
\item[(g.1)] $g(\cdot,\cdot)$ is continuous;
\item[(g.2)] $g(t,0)=0$ for all $t \geq 0$;
\item[(g.3)] $g(\cdot,\cdot)$ is global Lipschitz continuous with respect to the second variable, i.e., there exists a constant $L >0$ such that $\|g(t,x)-g(t,y)\|\leq L\|x-y\|$ for all $t\geq 0$ and $x,y\in \R^d$.
\end{itemize}
It is well known that the initial value problem for the fraction differential equation \eqref{introEq} has unique solution defined on the whole $\R_{\geq 0}$, for any  given initial value in $\R^d$ (see \cite[Theorem 2]{Baleanu}).
We will prove that there is no nontrivial solution of \eqref{introEq} converging to the origin with exponential rate.
\begin{lemma}\label{motivation}
Every nontrivial solution of \eqref{introEq} does not converge to the origin with exponential rate.
\end{lemma}
\begin{proof}
Due to the existence and uniqueness of solution to \eqref{introEq}, for any $x_0\ne 0$, the initial value problem \eqref{introEq} with the condition $x(0)=x_0$ has the unique solution $\Phi(\cdot,x_0)$ on the interval $[0,\infty)$. Assume that this solution converges to the origin with the exponential rate,
then
 there exist positive constants $\lambda$ and $T_1$ such that
\begin{equation}\label{t1}
\|\Phi(t,x_0)\|<\frac{1}{\exp{(\lambda t)}},\quad\text{for all}\; t\geq T_1.
\end{equation}
Take and fix a positive number $K>0$ satisfying $K\|x_0\|>1$. We recall here the notion of Mittag-Leffler functions; namely,
the \emph{Mittag-Leffler matrix function} $E_{\alpha,\beta}(A)$, for $\beta >0$ and a matrix $A\in\R^{d\times d}$ is defined as
\[
E_{\alpha,\beta}(A):=\sum_{k=0}^\infty \frac{A^k}{\Gamma(\alpha k+\beta)},\qquad E_{\alpha}(A):=E_{\alpha,1}(A),
\]
see, e.g., Diethelm~\cite{Kai}. In case $d=1$ the above formula gives definition of Mittag-Lefler function of a real variable.
 From the asymptotic behavior of the exponential functions and Mittag-Leffler functions, there is a constant $T_2>0$ such that
\begin{equation}\label{t2}
\frac{1}{ \exp{(\lambda t)}}<\frac{E_\alpha(-Lt^\alpha)}{K},\quad\text{for all}\; t\geq T_2.
\end{equation}
Put $T_0=\max\{T_1,T_2\}$. Using the equivalent integral form of \eqref{introEq}, by virtue of \eqref{t1} and \eqref{t2}, we have
\begin{align}
\notag \Gamma(\alpha)\|x_0\| & \leq \limsup_{t\to\infty}\left(L\int_0^{T_0}(t-s)^{\alpha-1}\|\Phi(s,x_0)\|\;ds+L\int_{T_0}^{t}(t-s)^{\alpha-1}\|\Phi(s,x_0)\|\;ds\right)\\
\notag  & \hspace*{-1.3cm}\leq \limsup_{t\to\infty}\left(L\sup_{s\in[0,T_0]}\|\Phi(s,x_0\|\int_0^{T_0}(t-s)^{\alpha-1}\;ds+\frac{L}{K}\int_{T_0}^t(t-s)^{\alpha-1}E_\alpha (-Ls^\alpha)\;ds\right)\\
& \leq \limsup_{t\to\infty}\frac{L}{K}\int_0^t(t-s)^{\alpha-1}E_\alpha (-Ls^\alpha)\;ds.\label{t3}
\end{align}
It is worth mentioning that $E_\alpha(-Lt^\alpha)$ is the solution of the initial value problem 
\begin{align*}
^{C}D^\alpha_{0+}x(t)&=-Lx(t),\quad t>0,\\
x(0)=1.
\end{align*}
Hence,
\[
E_\alpha(-Lt^\alpha)=1-\frac{L}{\Gamma(\alpha)}\int_0^t (t-s)^{\alpha-1}E_\alpha(-Ls^\alpha)\;ds,\quad \forall t\geq 1,
\]
and
\[
\lim_{t\to\infty}\int_0^t (t-s)^{\alpha-1}E_\alpha(-Ls^\alpha)\;ds=\frac{\Gamma(\alpha)}{L},
\]
a contradiction with \eqref{t3}. Therefore, there do not exist any nontrivial solution of \eqref{introEq} converging to the origin with the exponential rate. The proof is complete. 
\end{proof}
A closer look at the proof of  Lemma~\ref{motivation} allows us to have a even stronger statement on the decaying rate of solutions to fractional differential equations.

\begin{theorem}[Power rate decay of solution of FDEs] \label{thm.power-rate}
Any nontrivial solution of the FDE \eqref{introEq} cannot decay to 0 faster than $t^{-\alpha}$. More precisely, let $\Phi(\cdot,x_0)$ be an arbitrary solution of  the FDE \eqref{introEq} with initial value $\Phi(0,x_0) = x_0\not=0$ and $\beta >0$ be an arbitrary positive number satisfying $\beta > \alpha$, then 
$$
\limsup_{t\rightarrow+\infty}t^\beta\|\Phi(t,x_0) \| = +\infty.
$$
\end{theorem}

\begin{proof}
Assume, in contrary, that there exists an $\beta>\alpha$ such that
$$
\limsup_{t\rightarrow+\infty}t^\beta\|\Phi(t,x_0) \| = M <\infty.
$$
It suffices to use the arguments of the proof of Lemma \ref{motivation},  modifying the relations \eqref{t1} and \eqref{t2}   by changing $\exp(\lambda t)$ there to $t^\beta/(M+1)$, to derive a contradiction.
\end{proof}
\begin{remark}
Lemma \ref{motivation} remains true if we replace the strong condition of global Lipschitz property (g.3) by a weaker condition of local Lipschitz property of $g$ at the origin:
 \begin{itemize}
\item[(g.3')] There are positive constants $a>0, L>0$ such that $\|g(t,x)-g(t,y)\|\leq L\|x-y\|$ for all $t\geq 0$ and $x,y\in \R^d, \|x\|\leq a, \|y\|\leq a$.
\end{itemize}
  Similarly, nonuniform Lipschitz property (g.3') of $g$ suffices  for Theorem \ref{thm.power-rate}.
 \end{remark}
\subsection{Notions of stablity for FDE systems}
Consider the nonlinear fractional differential equation \eqref{introEq}
\begin{equation*}
^{C}D_{0+}^\alpha x(t)=g(t,x(t)),\quad t>0,
\end{equation*}
where $g:\R_{\geq 0}\times \R^d\rightarrow \R^d$ is continuous and satisfies the condition (g.1)-(g.2)-(g.3').  Since $g$ is local Lipschitz continuous, Theorem \ref{exist_result_1} and Theorem \ref{exist_result_2} imply unique existence of solution to the initial value problem \eqref{introEq}, $x(0) = x_0$ for $x_0\in \R^d, \|x\|\leq a$. Let $\Phi : I \times \R^d \rightarrow \R^d$ denote
the solution of \eqref{nonLinear}, $x(0) = x_0$, on its maximal interval of existence $I = [0, T_b (x_0))$ with $0 < T_{b} (x_0) \leq \infty $. We recall notions of stability and asymptotical stability of the trivial solution of \eqref{introEq} which is a direct application of the stability notions from classical ordinary differential equations theory to the FDE case, cf.\ \cite[Definition 7.2, p.~157]{Kai}. 
\begin{definition} 
\begin{itemize}
\item[(i)] The trivial solution of the nonlinear fractional differential equation \eqref{introEq} is called {\em stable} if for any $\varepsilon >0$ there exists $\delta=\delta(\varepsilon)>0$ such that for all $\|x_0\|<\delta$ we have $T_{b} (x_0)= \infty$ and $\|\Phi(t,x_0)\|<\varepsilon$ for all $t\ge 0$.
\item[(ii)] The trivial solution is called {\em asymptotically stable} if it is stable and there exists some $\tilde{\delta}>0$ such that $\lim_{t\to \infty}\|\Phi(t,x_0)\|=0$ whenever $\|x_0\|<\tilde{\delta}$.
\end{itemize}
\end{definition}
It is well known that there is a notion of exponential stability of solution of ordinary differential equations which related to the exponential rate of convergence to solutions. However, the results of Section~\ref{sec.power.rate} show that the non-trivial solution of FDEs cannot decay with exponential rate but at most power rate. Therefore, it make sense to investigate the power rate of decay of solution to FDEs.

In the equation \eqref{introEq} if $g(t,x)=Ax$ for all $t\geq 0$, $x\in \R^d$ and $A\in \R^{d\times d}$, then for any $x_0\in \R^d$, this system with the initial condition $x(0)=x_0$ has the unique solution $E_\alpha(t^\alpha A)x_0$ on the interval $[0,\infty)$. This suggests us to use the Mittag-Leffler function in establishing a suitable stability definition for systems of fractional- order. 

Motivated by Lemma \ref{motivation}, we now propose a new definition to characterize the convergent rate to the equilibrium points of solutions of FDEs. This is similar to that introduced by several authors (see Li \textit{et al.}~\cite{Li_1,Li_2}, Choi \textit{et al.}~\cite{Choi-14},  Stamova~\cite{Stamova_15}).
\begin{definition}\label{dfn.MLstability}
The equilibrium point $x^* = 0$ of  \eqref{introEq} is called {\em Mittag-Leffler stable} if there exist positive constants $\beta$, $m$ and $\delta$ such that
\begin{equation}
\sup_{t\geq 0} t^\beta \|\Phi(t,x_0)\| \leq m 
\end{equation}
for all $\|x_0\| \leq \delta $. 
\end{definition}

\begin{remark}
(i) Our definition of Mittag-Leffler stability is formulated in the form similar to the notion of exponential stability in the classical theory of ordinary differential equations. It reveals the power rate of decay of solutions of Mittag-Leffler stable  systems.

(ii) Due to the asymptotic behavior of the Mittag-Leffler function our definition is equivalent to the definition of Mittag-Leffler stability by several other authors (see Li \textit{et al.} \cite{Li_1,Li_2}, Choi \textit{et al.} \cite{Choi-14},  Stamova \cite{Stamova_15}).

(iii) In light of Theorem \ref{thm.power-rate} the parameter $\beta$ in the Definition \ref{dfn.MLstability} must satisfy $\beta\leq\alpha$.
\end{remark}

\section{Linearized Mittag-Leffler stability of fractional systems}\label{first_method}
In this section, we will develop a Lyapunov's first method to study the asymptotic behavior of solutions to FDEs.	Based on a variation of constant formula, properties of Mittag-Leffler function, Lyapunov-Perron approach and a new weighted norm which first appears in the literature, we obtain the Mittag-Leffler stability of fixed points to a class of nonlinear FDEs linearized about its equilibrium points.
\subsection{Formulation of the result}
Consider a nonlinear fractional differential equation of the following form
\begin{equation}\label{nonLinear}
^{C}D_{0+}^{\alpha}x(t)=Ax(t)+f(x(t)),
\end{equation}
where $A\in \R^{d\times d}$ and $f:\R^d\rightarrow \R^d$ is continuous on $\R^d$ and Lipschitz continuous in a neighborhood of the origin satisfying 
\begin{equation}\label{Lipschitz condition}
f(0)=0\quad\hbox{and}\quad\lim_{r\to 0} \ell_f(r)=0,
\end{equation}
in which 
\[
\ell_f(r):=\sup_{x,y\in B_{\R^d}(0,r)}\frac{\|f(x)-f(y)\|}{\|x-y\|}.
\]
Furthermore, let $\lambda_1,\dots,\lambda_n$ denote the eigenvalues of $A$. Suppose that 
\begin{equation}\label{stable_cond}
\lambda_i\in \Lambda_\alpha^s:=\left\{\lambda\in\C\setminus\{0\}:|\hbox{arg}(\lambda)|>\frac{\alpha \pi}{2}\right\},\qquad i=1,\dots,n.
\end{equation}
Our task is to study the asymptotic behavior of solutions to \eqref{nonLinear} around the origin. In \cite{CST_16}, the authors give a linearized stability theorem for the trivial solution of \eqref{nonLinear} as follows.
\begin{theorem}[see {\cite[Theorem 3.1]{CST_16}}]
Assume that $A$ satisfies the condition \eqref{stable_cond} and $f(\cdot)$ satisfies the condition \eqref{Lipschitz condition}. Then the trivial solution of the system \eqref{nonLinear} is asymptotically stable.
\end{theorem}
After work \cite[Theorem 3.1]{CST_16} a natural question now  arises: what is the rate of convergence to the origin for solutions of the nonlinear FDE \eqref{nonLinear}?  As shown above (see Theorem \ref{thm.power-rate}) the trivial solution of fractional-order systems can not be exponentially stable. Hence, the best rate of convergence one may expect is the polynomial rate, and one of our main contributions is the following result on Mittag -Leffler stability of FDEs.

\begin{theorem}[Lyapunov's first method for Mittag-Leffler stability]\label{main result}
Assume that $A$ satisfies \eqref{stable_cond} and $f(\cdot)$ satisfies the condition \eqref{Lipschitz condition}. Then the trivial solution of the system \eqref{nonLinear} is Mittag-Leffler stable.
\end{theorem}

To prove Theorem \ref{main result} we need the lemmas below.
\begin{lemma}\label{key_est}
\begin{itemize}
\item [(i)] For any $\lambda\in \Lambda_\alpha^s$, there exists a constant $C_1>0$ such that
\[
|E_\alpha(\lambda t^\alpha)|\leq C_1 E_\alpha(-t^\alpha),\quad \forall t\geq 0.
\]
\item [(ii)] There is a constant $C_2>0$ such that
\[
t^\alpha \int_0^t (t-s)^{\alpha-1}E_{\alpha,\alpha}(-(t-s)^\alpha)s^{-\alpha}\;ds\leq C_2,\quad \forall t\geq 0.
\]
\end{itemize}
\end{lemma}
\begin{proof}
\noindent (i) The proof of this statement is obvious.\\

\noindent (ii) The proof is deduced by using \cite[Formula (1.100)]{Podlubny} and the asymptotic behavior of Mittag-Leffler function $E_\alpha (-t^\alpha)$.
\end{proof}
\begin{lemma}
Let $\lambda\in \Lambda_\alpha^s$. Then, there exists a positive constant $C_3$ such that
\[
\int_0^\infty \tau^{\alpha-1}|E_{\alpha,\alpha}(\lambda \tau^\alpha)|\,d\tau<C_3.
\]
\end{lemma}
\begin{proof}
See \cite[Theorem 3(ii)]{Cong_Tuan_2018}.
\end{proof}
\subsection{Proof of Theorem \ref{main result}}
We follow the approach in our preceding paper \cite{CST_16} to complete the proof of Theorem \ref{main result}. This proof contain two main steps:
\begin{itemize}
\item Transformation of the linear part: we transform the linear part of \eqref{nonLinear} to a matrix which is "very close" to a diagonal matrix. This technical step help us to reduce the difficulty in the estimation of the matrix valued Mittag-Leffler function in the next step.
\item Construction of an appropriate Lyapunov-Perron operator: we will present a family of operators with the property that any solution of the nonlinear system \eqref{nonLinear} can be interpreted as a fixed point of these operators. Furthermore, these operators are contractive in a suitable space hence the fixed points of these operators can be estimated.
\end{itemize}
\bigskip
{\bf Transformation of the linear part}
\bigskip

By virtue \cite[Theorem 6.37, pp.~146]{Shilov}, we can find a nonsingular matrix $T\in\C^{d\times d}$ such that
\[
T^{-1}A T=\hbox{diag}(A_1,\dots,A_n),
\]
where for $i=1,\dots,n$ the block $A_i$ has the form
\[
A_i=\lambda_i\, \id_{d_i\times d_i}+\delta_i\, N_{d_i\times d_i},
\]
with $\lambda_i$ is an eigenvalue, $\delta_i\in\{0,1\}$ and the nilpotent matrix $N_{d_i\times d_i}$ is given by
\[
N_{d_i\times d_i}:=
\left(
      \begin{array}{*7{c}}
      0  &     1         &    0      & \cdots        &  0        \\
        0        & 0    &    1     &   \cdots      &              0\\
        \vdots &\vdots        &  \ddots         &          \ddots &\vdots\\
        0 &    0           &\cdots           &  0 &          1 \\

        0& 0  &\cdots                                          &0         & 0 \\
      \end{array}
    \right)_{d_i \times d_i}.
\]
Let $\eta$ be an arbitrary but fixed positive number. Applying the transformation $P_i:=\textup{diag}(1,\eta,\dots,\eta^{d_i-1})$ leads to
\begin{equation*}
P_i^{-1} A_i P_i=\lambda_i\, \id_{d_i\times d_i}+\eta_i\, N_{d_i\times d_i},
\end{equation*}
$\eta_i\in \{0,\eta\}$. Hence, under the transformation $y:=(TP)^{-1}x$ system \eqref{nonLinear} becomes
\begin{equation}\label{NewSystem}
^{C}D_{0+}^\alpha y(t)=\hbox{diag}(J_1,\dots,J_n)y(t)+h(y(t)),
\end{equation}
where $J_i:=\lambda_i \id_{d_i\times d_i}$ for $i=1,\dots,n$ and the function $h$ is given by
\begin{equation*}\label{Eq3}
h(y):=\text{diag}(\eta_1N_{d_1\times d_1},\dots,\eta_nN_{d_n\times d_n})y+(TP)^{-1}f(TPy).
\end{equation*}
\begin{remark}[see {\cite[Remark 3.2]{CST_16}}]\label{Remark1}
The map $x\mapsto \text{diag}(\eta_1N_{d_1\times d_1},\dots,\eta_nN_{d_n\times d_n})x$ is a Lipschitz continuous function with Lipschitz constant $\eta$. Thus, by \eqref{Lipschitz condition} we have
\[
h(0)=0,\qquad \lim_{r\to 0}\ell_h(r)=
\left\{
\begin{array}{ll}
\eta & \hbox{if there exists } \eta_i=\eta,\\[1ex]
0 & \hbox{otherwise}.
\end{array}
\right.
\]
\end{remark}
\begin{remark}\label{Remark2}
The type of stability of the trivial solution to equations \eqref{nonLinear} and \eqref{NewSystem} is the same, i.e., they are both stable (asymptotic/Mittag-Leffler stable) or not stable (asymptotic/Mittag-Leffler stable).
\end{remark}

\bigskip
\noindent {\bf Construction of an appropriate Lyapunov-Perron operator}
\bigskip

We now concentrate only on the equation \eqref{NewSystem} and introduce a Lyapunov-Perron operator associated with this equation.

For any $x=(x^1,\dots,x^n)\in \C^{d}=\C^{d_1}\times\dots\times \C^{d_n}$, the operator $\mathcal{T}_{x}: C_\infty([0,\infty),\C^d)\rightarrow C_\infty([0,\infty),\C^d)$ is defined by
\[
(\mathcal{T}_{x}\xi)(t)=((\mathcal{T}_{x}\xi)^1(t),\dots,(\mathcal{T}_{x}\xi)^n(t))\qquad\hbox{for } t\in\R_{\geq 0},
\]
where for $i=1,\dots,n$
\begin{eqnarray*}
(\mathcal{T}_{x}\xi)^i(t)
&=&
E_\alpha(t^\alpha J_i)x^i+\\
&&
\int_0^t (t-\tau)^{\alpha-1}E_{\alpha,\alpha}((t-\tau)^\alpha J_i)h^i(\xi(\tau))\;d\tau,
\end{eqnarray*}
is called the \emph{Lyapunov-Perron operator associated with \eqref{NewSystem}}. The relationship between a fixed point of the operator $\mathcal{T}_x(\cdot)$ and a solution to the equation \eqref{NewSystem} is described in the lemma below.
\begin{lemma}\label{Var_Const_Form}
Consider \eqref{NewSystem} and assume that the function $h(\cdot)$ is global Lipschitz continuous. Let $x\in\C^{d}$ be arbitrary and $\xi:\R_{\geq 0}\rightarrow \C^{d}$ be a continuous function satisfying that $\xi(0)=x$. Then, the following statements are equivalent:
\begin{itemize}
\item [(i)] $\xi$ is a solution of \eqref{NewSystem} satisfying the initial condition $x(0)=x$.
\item [(ii)] $\xi$ is a fixed point of the operator $\mathcal {T}_{x}$.
\end{itemize}
\end{lemma}
\begin{proof}
The assertion follows from the variation of constants formula for fractional differential equations, see e.g., \cite{Cong_Tuan_2016}.
\end{proof}
Our novel contribution in the present work is to combine the approach in \cite{CST_16} and a new weighted norm as follows. In $C_\infty([0,\infty),\C^d)$ we define a function $\|\cdot\|_w$ by
\[
\|x\|_w=\max\{\sup_{t\in [0,1]}\|x(t)\|,\;\sup_{t\geq 1}t^\alpha\|x(t)\|\}.
\]
Then $C_w:=\{x\in C_\infty([0,\infty),\C^d):\|x\|_w<\infty\}$ is also a Banach space with the norm $\|\cdot\|_w$.

Next, we provide some estimates concerning the operator $\mathcal{T}_{x}$ in the space $C_w$. 
\begin{proposition}\label{Prp2} Consider system \eqref{NewSystem} and suppose that
\[
\lambda_i\in \Lambda_\alpha^s,\qquad i=1,\dots,n,
\]
where $\lambda_1,\dots,\lambda_n$ are eigenvalues of $A$. Then, there exists a constant $C(\alpha,A)$ depending on $\alpha$ and $\lambda:=(\lambda_1,\dots,\lambda_n)$ such that for all $x,\widehat x\in\C^{d}$ and $\xi,\widehat\xi\in C_w$ the following inequality holds
\begin{eqnarray*}\label{Contraction}
\notag \|\mathcal T_{x}\xi-\mathcal T_{\widehat x}\widehat\xi\|_w
& \leq &
\max_{1\le i \le n}\{\sup_{t\in [0,1]}|E_\alpha(\lambda_i t^\alpha)|+\sup_{t\geq 1}t^\alpha|E_\alpha(\lambda_i t^\alpha)|\} \|x-\widehat x\|\\
&+&C(\alpha,A)\;
\ell_h(\max\{\|\xi\|_\infty,\|\widehat\xi\|_\infty\})\|\xi-\widehat\xi\|_w.
\end{eqnarray*}
Consequently, $\mathcal T_{x}$ considered as an operator on the Banach space $C_w$ endowed with the norm $\|\cdot\|_w$ is well-defined and
\begin{equation}\label{well_defined}
\|\mathcal T_{x}\xi-\mathcal T_{x}\widehat\xi\|_w \leq C(\alpha,A)\; \ell_h(\max\{\|\xi\|_\infty,\|\widehat\xi\|_\infty)\}\|\xi-\widehat\xi\|_w.
\end{equation}
\end{proposition}
\begin{proof}
For $i=1,\dots,n$, we get
\begin{align*}
\|(\mathcal{T}_x \xi)^i(t)-(\mathcal{T}_{\widehat{x}} \widehat{\xi})^i(t)\|
&\le\|x-\widehat x\||E_\alpha (\lambda_i t^\alpha)|\\
&\hspace{-3cm}+\ell_h(\max\{\|\xi\|_\infty,\|\widehat{\xi}\|_\infty\})
\int_0^t(t-\tau)^{\alpha-1} |E_{\alpha,\alpha}(\lambda_i (t-\tau)^\alpha)|\|(\xi-\widehat{\xi})(\tau)\|\;d\tau.
\end{align*}
In the case $t\in [0,1]$, we have
\begin{align}
\notag \sup_{t\in [0,1]}\|(\mathcal T_{x}\xi-\mathcal T_{x}\widehat\xi)^i(t)\|&\leq \sup_{t\in [0,1]}|E_\alpha(\lambda_i t^\alpha)|\|x-\widehat{x}\|\\
&\hspace{-3cm}+\ell_h(\max\{\|\xi\|_\infty,\|\widehat{\xi}\|_\infty\})\int_0^\infty u^{\alpha-1}|E_{\alpha,\alpha}(-\lambda_i u^\alpha)|\;du\;\|\xi-\widehat\xi\|_w.\label{est1}
\end{align}
Furthermore,
\begin{align}
\notag \sup_{t\geq 1}t^\alpha\|(\mathcal T_{x}\xi-\mathcal T_{x}\widehat\xi)^i(t)\|&\leq \sup_{t\geq 1}t^\alpha |E_\alpha(\lambda_i t^\alpha)|\;\|x-\widehat{x}\|\\
&\hspace{-4cm}+C_{\lambda_i}\ell_h(\max\{\|\xi\|_\infty,\|\widehat{\xi}\|_\infty\})\sup_{t\geq 1}t^\alpha\int_0^t (t-\tau)^{\alpha-1}E_{\alpha,\alpha}(-(t-\tau)^\alpha)\tau^{-\alpha}\;d\tau \|\xi-\widehat\xi\|_w,\label{est2}
\end{align}
where $C_{\lambda_i}$ is a constant chosen as in Lemma \ref{key_est} (i). Now by combining Lemma \ref{key_est}, \eqref{est1} and \eqref{est2}, we have
\begin{align*}
\|\mathcal T_x \xi)-\mathcal T_{\widehat x} \widehat \xi\|_w&\leq  \max_{1\le i \le n}\{\sup_{t\in [0,1]}|E_\alpha(\lambda_i t^\alpha)|+\sup_{t\geq 1}t^\alpha|E_\alpha(\lambda_i t^\alpha)|\} \|x-\widehat x\|\\
&+C(\alpha,A)\;
\ell_h(\max\{\|\xi\|_\infty,\|\widehat\xi\|_\infty\})\|\xi-\widehat\xi\|_w,
\end{align*}
where 
\begin{align*}
C(\alpha,A)&:=\max_{1\leq i\leq n}\int_0^\infty u^{\alpha-1}|E_{\alpha,\alpha}(\lambda_i u^\alpha)|\;du\\
&\hspace{1cm}+C_\lambda\sup_{t\geq 1}t^\alpha\int_0^t (t-\tau)^{\alpha-1}E_{\alpha,\alpha}(-(t-\tau)^\alpha)\tau^{-\alpha}\;d\tau
\end{align*}
with $C_\lambda:=\max\{C_{\lambda_1},\dots,C_{\lambda_n}\}$. The proof is complete.
\end{proof}
We have showed that the Lyapunov-Perron operator $\mathcal{T}_x(\cdot)$ is well-defined and Lipschitz continuous and that the constant $C(\alpha,A)$ is independent of the constant $\eta$. From now on, we choose and fix the constant $\eta$ as $\eta:=\frac{1}{2C(\alpha,A)}$. 
\begin{lemma}\label{lemma6}
The following statements hold:
\begin{itemize}
\item [(i)] There exists $r>0$ such that
\begin{equation}\label{Eq7a}
q:=C(\alpha,A)\;  \ell_h(r) < 1.
\end{equation}
\item [(ii)] Choose and fix $r>0$ satisfying \eqref{Eq7a}. Define
\begin{equation}\label{Eq7b}
r^*:=\frac{r(1-q)}{\max_{1\le i\le n}\{\sup_{t\in [0,1]}|E_\alpha(\lambda_it^\alpha)|+\sup_{t\geq 1}t^\alpha|E_\alpha(\lambda_i t^\alpha)|\}}.
\end{equation}
Let  $B_{C_{w}}(0,r):=\{\xi\in C_\infty([0,\infty),\C^d):\left||\xi|\right|_w\le r\}$. Then, for any $x\in B_{\C^{d}}(0,r^*)$ we have $\mathcal T_{x} (B_{C_{w}}(0,r))\subset B_{C_{w}}(0,r)$ and
\begin{equation*}\label{LipschitzContinuity}
\|\mathcal T_{x}\xi-\mathcal T_{x}\widehat {\xi}\|_w
\leq
q\|\xi-\widehat{\xi}\|_w\quad\hbox{ for all } \xi,\widehat{\xi}\in B_{C_{w}}(0,r).
\end{equation*}
\end{itemize}

\end{lemma}
\begin{proof}
(i) By Remark \ref{Remark1}, $\lim_{r\to 0}\ell_h(r)\leq \eta$. Since $\eta C(\alpha,A)=\frac{1}{2}$, the assertion (i) is proved.

(ii) Let $x\in B_{\C^{d}}(0,r^*)$ and $\xi\in B_{C_w}(0,r)$. According to \eqref{Contraction} in Proposition \ref{Prp2}, we obtain that
\begin{eqnarray*}
\|\mathcal T_{x}\xi\|_w
& \leq & \;\max_{1\le i \le n}\{\sup_{t\in [0,1]}|E_\alpha(\lambda_i t^\alpha)|+\sup_{t\geq 1}t^\alpha|E_\alpha(\lambda_i t^\alpha)|\}\|x\|+ C(\alpha,A)\,\ell_{h}(r)\|\xi\|_{w}\\
& \leq& \;(1-q)r+qr,
\end{eqnarray*}
which proves that $\mathcal T_{x}(B_{C_{w}}(0,r))\subset B_{C_{w}}(0,r)$. Furthermore, by  Proposition \ref{Prp2} and part (i) for all $x\in B_{\C^d}(0,r^*)$ and $\xi,\widehat{\xi}\in B_{C_w}(0,r)$ we have
\begin{eqnarray*}
\|\mathcal T_{x}\xi-\mathcal T_{x}\widehat{\xi}\|_w
&\leq&
C(\alpha,A)\ell_{h}(r)\;\|\xi-\widehat{\xi}\|_w\\[1.5ex]
&\leq &
q\|\xi-\widehat{\xi}\|_w,
\end{eqnarray*}
which ends the proof.
\end{proof}
\begin{proof}[Proof of Theorem \ref{main result}]
Due to Remark \ref{Remark2}, it is sufficient to prove the Mittag-Leffler  stability for the trivial solution of \eqref{NewSystem}. For this purpose, let $r^*$ be defined as in \eqref{Eq7b}. Let $x\in B_{\C^d}(0,r^*)$ be arbitrary. Using Lemma \ref{lemma6} and the Contraction Mapping Principle, there exists a unique fixed point $\xi \in B_{C_w}(0,r)$ of $\mathcal{T}_x$. This fixed point is also the unique solution of \eqref{NewSystem} with the initial condition $\xi(0)=x$ (see Lemma \ref{Var_Const_Form}). Together existence and uniqueness of solutions for initial value problems for the equation \eqref{NewSystem} in a neighborhood of the origin, this shows that the solution $0$ is stable in the Lyapunov's sense. Moreover,
\[
\sup_{t\geq 0}t^\alpha \|\xi(t)\|\leq r,
\]
which shows that the trivial solution of \eqref{NewSystem} is Mittag-Leffler stable. The proof is complete.
\end{proof}
\begin{remark}
In \cite[Theorem 3.1]{CST_16}, we proved that the trivial solution to \eqref{nonLinear} is asymptotically stable. However, we did not know decay rate of non-trivial solutions to this equation. Now by Theorem \ref{main result} this question is answered fully. Namely, in the proof of Theorem \ref{main result} we showed the convergence rate of solutions around the equilibrium as $t^{-\alpha}$.
\end{remark}
\begin{remark}
In the case $A=0$, we can not use the linearization method around an equilibrium point to analyze the Mittag-Leffler stability of \eqref{nonLinear}. To overcome this obstacle, in Section~\ref{sec.second-Lyapunov} we will develop the Lyapunov's second method for fractional differential equations.
\end{remark}
\section{Lyapunov's second method and Mittag-Leffler stability}\label{sec.second-Lyapunov}
In this section, we develop a Lyapunov's second method for fractional-order systems. Our approach is based on a comparison principle for FDE and an inequality concerning with fractional derivative of convex functions. For this purpose, we need the following preparation results. 
\begin{lemma}\label{lemma22}
Let $m:[0,T]\rightarrow \R$ be continuous and Caputo derivative $^{C}D^\alpha_{0+}m$ exists on the interval $(0,T]$. If there exists $t_0\in (0,T]$ such that
\[
m(t)\leq 0\quad \forall t\in [0,t_0)\quad\text{and}\quad m(t_0)=0,
\]
then $^{C}D^\alpha_{0+}m(t_0)\geq 0$.
\end{lemma}
\begin{proof}
Using the same arguments as in \cite[Lemma 2.1]{vastala}
\end{proof}
Based on arguments as in \cite[Theorem 2.3]{vastala}, we obtain the following comparison proposition.
\begin{proposition}\label{comparision}
Let $L:\R\rightarrow \R$ be continuous and non-increasing (it means that for $x_1\leq x_2$ then $L(x_1)\geq L(x_2)$, $m_1:[0,T]\rightarrow \R$, $m_2:[0,T]\rightarrow \R$ be continuous. Assume that $^{C}D^\alpha_{0+}m_1$, $^{C}D^\alpha_{0+}m_1$ exist on $(0,T]$. If \begin{align}
 ^{C}D^\alpha_{0+}m_1(t)&\geq L(m_1(t)),\quad t\in(0,T],\quad m_1(0) \geq m_0,\label{inequal}\\
 ^{C}D^\alpha_{0+}m_2(t)&\leq L(m_2(t)),\quad t\in(0,T],\quad m_2(0) \leq m_0,\label{inequal1}
\end{align}
then $m_1(t)\geq m_2(t)$ for all $t\in (0,T]$.
\end{proposition}
\begin{proof}
Assume first without loss of generality that one of the inequalities in \eqref{inequal} and \eqref{inequal1} is strict,
say $^{C}D^\alpha_{0+}m_2(t)<L(m_2(t))$ and $m_2(0)<m_0\leq m_1(0)$. We will show that
$m_2(t)<m_1(t)$ for all $t\in [0,T]$. Suppose, to the contrary, that there exists $t_0$ such that $0 < t_0 \leq T$ for which
$m_2(t_0)=m_1(t_0)$ and $m_2(t)<m_1(t)$ for all $t\in [0,t_0)$. Set $m(t)=m_2(t)-m_1(t)$ it follows that $m(t_0)=0$ and $m(t)<0$ for $t\in [0,t_0)$. Then by
hypothesis and Lemma \ref{lemma22} we have that $^{C}D^\alpha_{0+}m(t_0)\geq 0$. Therefore, since
 $m_2(t_0)=m_1(t_0)$ we get
$$
L(m_2(t_0))>\; ^{C}D^\alpha_{0+}m_2(t_0)\geq\; ^{C}D^\alpha_{0+}m_1(t_0)\geq L(m_1(t_0)) = 
L(m_2(t_0)),
$$
which is a contradiction. Consequently, $m_2(t)<m_1(t)$ for all $t\in [0,T]$.
Now assume that the inequalities in \eqref{inequal} are non-strict. We will show that $m_2(t)\leq m_1(t)$ for all $t\in [0,T]$. Set $m_1^\varepsilon (t)=m_1(t)+\varepsilon \lambda (t)$ where $\varepsilon>0$ and $\lambda(t)=E_\alpha(t^\alpha)$. Noting that $\lambda(\cdot)$ is positive, we have $m_1^\varepsilon(t)>m_1(t)$ and, since $L(\cdot)$ is non-increasing 
\begin{align}\label{temp}
\notag ^{C}D^\alpha_{0+} m_1^\varepsilon (t)&=\;^{C}D^\alpha_{0+}m_1(t)+\varepsilon \lambda(t)\\
\notag &\;\geq L(m_1(t))+\varepsilon \lambda(t)\\
\notag &=\;L(m_1^\varepsilon (t))+L(m_1(t))-L(m_1^\varepsilon (t))+\varepsilon \lambda(t)\\
\notag &\;\geq L(m_1^\varepsilon (t))+\varepsilon \lambda(t)\\
&>\;L(m_1^\varepsilon (t)),\quad t\in(0,T].
\end{align}
Due to \eqref{temp}, applying now the result above for strict inequalities to $m_1^\varepsilon(t)$ and $m_2(t)$, we get that $m_2(t)<m_1^\varepsilon(t)$ for all $t\in[0,T]$.
Consequently, letting  $\varepsilon\to 0$ we get that $m_2(t)\leq m_1(t)$ for all $t\in [0,T]$. The proof is complete.
\end{proof}
\begin{remark}
Proposition \ref{comparision} improved \cite[Theorem 2.3]{vastala} in the way that we do not need to require continuous differentiability of $m_1(\cdot), m_2(\cdot)$, and Lipschitz property of $L(\cdot)$. This  improvement is very useful for our purpose in the next steps.
\end{remark}
\subsection{Lyapunov's second method for fractional differential equations}

Let $D\subset \R^d$ is an open set and $0\in D$. Consider the following fractional order equation with the order $\alpha\in (0,1)$:
\begin{equation}\label{Eq_Ex}
^{C}D^\alpha_{0+}x(t)=f(x(t)),\quad \text{for}\; t\in (0,\infty),
\end{equation}
where $f:D\rightarrow \R^d$ satisfies the following conditions:
\begin{itemize}
\item[(f.1)] $f(0)=0$;
\item[(f.2)] the function $f(\cdot)$ is Lipchitz continuous in a neighborhood of the origin.
\end{itemize}
The main result in this section is the following theorem.
\begin{theorem}[Lyapunov's second method for Mittag-Leffler stability]\label{main_res}
Consider the equation \eqref{Eq_Ex}. Assume there is a function $V:\R^d\rightarrow \R_+$ satisfying
\begin{itemize}
\item[(V.1)] the function $V$ is convex and differentiable on $\R^d$;
\item[(V.2)] there exist constants $a,b,C_1,C_2,r>0$ such that $$C_1\|x\|^a\leq V(x)\leq C_2 \|x\|^b$$
for all $x\in B_{\R^d}(0,r)$; 
\item[(V.3)] there are constants $C_3, c\geq 0$ such that 
$$\langle \nabla V(x),f(x)\rangle\leq -C_3\|x\|^c$$
for all $x\in B_{\R^d}(0,r)$.
\end{itemize}
Then,
\begin{itemize}
\item[(a)] the trivial solution of \eqref{Eq_Ex} is stable if $C_3= 0$;
\item[(b)] the trivial solution of \eqref{Eq_Ex} is Mittag-Leffler stable if $C_3>0$.
\end{itemize}
\end{theorem}
\begin{proof}
\noindent (a) See the proof of \cite[Theorem 3(a)]{Tuan_Hieu}.\smallskip

\noindent (b) Due to the fact that the trivial solution to \eqref{Eq_Ex} is stable, for any $\varepsilon>0$ 
 there exists $\delta>0$ such that every solution $\varphi(t,x_0)$ to \eqref{Eq_Ex} with $\|x_0\|<\delta$ satisfies $\|\varphi(t,x_0)\|<\varepsilon$ for all $t\geq 0$. Moreover, from \cite[Theorem 2]{Tuan_Hieu} and the conditions (V.2) and (V.3), we have
\begin{align*}
^{C}D^\alpha_{0+}V(\varphi(t,x_0))&\leq \langle\nabla V(\varphi(t,x_0)),^{C}D^\alpha_{0+}\varphi(t,x_0)\rangle\\
&\leq -C_3\|\varphi(t,x_0)\|^c\\
&\leq -\frac{C_3}{C_2^{c/b}}(V(\varphi(t,x_0)))^{c/b},\quad \forall t\geq 0.
\end{align*}
Put $A:=-\frac{C_3}{C_2^{c/b}}$, $p:=\frac{c}{b}$ and consider the following initial value problem 
\begin{equation}\label{tam_1}
\begin{cases}
^{C}D^\alpha_{0+}y(t)=Ay^p(t),\quad t>0,\\
y(0)=V(x_0)>0.
\end{cases}
\end{equation}
Then $V(\varphi(\cdot,x_0))$ is a sub-solution of \eqref{tam_1} (for the definition of sub-solution see \cite{zacher}). Furthermore, from the construction of a super-solution to \eqref{tam_1} (see \cite[p. 333]{zacher}), we can find a super-solution $w$ of \eqref{tam_1} on $[0,\infty)$ defined by
\[
w(t)=\begin{cases}
V(x_0),\quad t\in[0,t_1],\\
Ct^{-\frac{\alpha}{p}},\quad t\geq t_1,
\end{cases}
\]
where $C=V(x_0)t_1^{\frac{\alpha}{p}}$ and
\[
t_1^\alpha=\frac{V(x_0)^{1-p}}{-A}\left(\frac{2^\alpha}{\Gamma(1-\alpha)}+\frac{\alpha}{p}\frac{2^{\alpha+\frac{\alpha}{p}}}{\Gamma(2-\alpha)}\right).
\]
Now using the comparison proposition \ref{comparision}, we obtain 
\[
V(\varphi(t,x_0))\leq w(t),
\quad \forall t\geq 0.
\]
This implies that for any $x_0\in B_{\R^d}(0,\delta)\setminus\{0\}$, there exists a constant $d>0$ such that
\[
\|\varphi(t,x_0)\|\leq \left(\frac{1}{C_1}V(\varphi(t,x_0))\right)^{1/a}\leq \left(\frac{d}{C_1(1+t^{\alpha/p})}\right)^{1/a}
\]
for all $t\geq 0$. Note that from the existence and uniqueness of the solution to \eqref{Eq_Ex}, if $x_0=0$ then $\varphi(\cdot,0)=0$. So, the trivial solution to the original system \eqref{Eq_Ex} is Mittag-Leffler stable. The proof is complete.
\end{proof}
\begin{remark}
 (i) Theorem \ref{main_res} is still true if we replace the condition of global convex and differentiable property (V.1) by a condition of local convex and differentiable property in a neighborhood of the origin.

 (ii) Theorem \ref{main_res} is a new contribution in the theory of Lyapunov's second method for fractional differential equations. It improves and strengthens a recent result by Tuan and Hieu \cite[Theorem 3]{Tuan_Hieu}. In particular, we removed the condition $c>b$ in the statement of \cite[Theorem 3(c)]{Tuan_Hieu}. Moreover, we obtained the Mittag-Leffler stability of the trivial solution instead of the weakly asymptotic stability.
\end{remark}

\subsection{Illustrative examples}

\begin{example}[Simple nonlinear one-dimensional FDE]\label{example2}
 Consider the nonlinear one-dimensional FDE of order $0<\alpha<1$ which is nonlinear of order $\beta\geq 1$:
\begin{equation}\label{Ex2}
^{C}D_{0+}^{\alpha}x(t)=f(x(t)), \quad x(0)=x_0,
\end{equation}
where
\begin{equation}\label{Ex2'}
f(x) :=  \begin{cases} -x^\beta, \quad\hbox{if}\quad x\geq 0,\\
|x|^\beta,\quad\hbox{if}\quad x<0.
\end{cases}
\end{equation}
It is obvious that the function $f(\cdot)$ is local Lipschitz continuous at the origin. Choosing the function $V(x)=x^2$, $x\in \R$. This function satisfies the conditions (V.1), (V.2) (with $C_1=C_2=1$ and $a=b=2$), and (V.3) (with $C_3=2$ and $c=1+\beta$). Thus, from Theorem \ref{main_res}, the trivial solution to \eqref{Ex2} is Mittag-Leffler stable. More precisely, from the proof of Theorem \ref{main_res}, we see that the non-trivial solutions of \eqref{Ex2} converge to the origin with the rate at least $t^{-\alpha/(1+\beta)}$ as $t\rightarrow\infty$. A special case of \eqref{Ex2} when $\beta=3$ was studied by Li \textit{et al.} \cite[Example 14]{Li_1}, Shen \textit{et al.} \cite[Remark 11]{Shen}, Zhou \textit{et al.} \cite{Zhou}, where they tried to prove asymptotic stability of \eqref{Ex2}. However, their proof is not correct, see Tuan and Hieu \cite[Remark 3]{Tuan_Hieu} for details. Our method now solves this problem completely: we showed that the trivial solution of \eqref{Ex2} is Mittag-Leffler stable, hence asymptotically stable.
\end{example}
\begin{example}[A more complicated nonlinear one-dimensional FDE]
Consider an equation in form
\begin{equation}\label{Ex4}
^{C}D^\alpha_{0+}x(t)=-x^3+g(x(t)),\quad t>0,
\end{equation}
where $g:\R\rightarrow \R$ is differentiable at the origin and satisfies
\[
g(0)=0,\quad\lim_{x\to 0}\frac{g(x)}{x^3}=0.
\]
\end{example}
Choosing the Lyapunov candidate function $V(x)=x^2$ for $x\in \R$ and $r>0$ such that
\[
2x(-x^3+g(x))\leq -x^4,\quad \forall x\in B_{\R}(0,r).
\] 
Then the conditions of Theorem \ref{main_res} are satisfied for $C_1=C_2=1$, $a=b=2$, $C_3=1$ and $c=4$. Thus, the trivial solution of \eqref{Ex4} is Mittag-Leffler stable.

\begin{example}[Higher dimensional nonlinear FDE]
Consider a two dimensional fractional-order nonlinear system
\begin{equation}\label{Ex5}
^{C}D^\alpha_{0+}x(t)=f(x(t)),\quad t>0,
\end{equation}
where $f(x)=(-x_1^3+x_2^4,-x_2^3-x_2x_1^2)^{\rm T}$ for any $x=(x_1,x_2)\in \R^2$. In this case we choose the Lyapunov candidate function $V(x)=\|x\|^2=x_1^2+x_2^2$ for $x=(x_1,x_2)\in \R^2$ and $r>0$ such that
\[
\langle(2x_1,2x_2),(-x_1^3+x_2^4,-x_2^3-x_1^2 x_2)\rangle\leq -x_1^4-x_2^4
\] 
for all $x=(x_1,x_2)\in B_{\R^2}(0,r)$. The function $V(\cdot)$ now satisfies the conditions $\textup{(V.1)}$, $\textup{(V.2)}$ and $\textup{(V.3)}$ in Theorem \ref{main_res} for $a=b=2$, $c=4$, $C_1=C_2=1$ and $C_3=1$. Hence, the trivial solution of \eqref{Ex5} is Mittag-Leffler stable.
\end{example}

\section{Relation between Lipschitz condition, stability and speed of decay, separation of trajectories to Caputo FDEs}\label{final_sec}

We first present here several examples of Caputo FDEs of  various kinds of stability to illustrate the stability notions given in Section~\ref{sec.power.rate}. It is obvious that Mittag-Leffler stability is stronger than asymptotic stability.
 
 \begin{example}[Linear autonomous FDE]
 Let us consider a linear autonomous FDE of order $\alpha\in (0,1)$:
\begin{equation}\label{Ex1}
^{C}D_{0+}^{\alpha}x(t)=Ax(t),\quad A= \hbox{diag}(a_1,\ldots, a_d), \quad a_i<0, i=1,\ldots, d.
\end{equation}
This FDE is solvable explicitly and and its solutions are of the form\break 
 $\hbox{diag}(E_\alpha(a_1t^\alpha),\ldots, E_\alpha(a_d t^\alpha))x_0, \; x_0\in\R^d$ (see Diethelm~\cite[Theorem 7.2]{Kai}).
It is easy to see that the trivial solution of \eqref{Ex1} is Mittag-Leffler stable and all non-trivial solutions have decay rate $t^{-\alpha}$.
 \end{example}
 
Unlike the linear autonomous case, solution to nonlinear FDEs may have decay rate smaller or bigger than order of the equations. The FDE \eqref{Ex2} treated in Example \ref{example2} is a nonlinear FDE with solutions decaying to 0 with rate slower that $t^{-\alpha}$. Actually we show in Example \ref{example2}, using Theorem \ref{main_res} that the decay rate of nontrivial solutions to the FDE \eqref{Ex2} is at least
$t^{-\alpha/(1+\beta)}$ as $t\rightarrow\infty$. An application of the result of Vergara and Zacher \cite[Theorem 7.1, p. 334]{zacher} shows that decay rate of nontrivial solutions to the FDE \eqref{Ex2} is $\alpha/\beta<\alpha$ for $\beta>1$. Our next example is a nonlinear FDE with nontrivial solutions having decay rate bigger than the order of the equation.

\begin{example}[Nonlinear one-dimensional FDE with non Lipschitz right-hand side]
Consider the nonlinear one-dimensional FDE of order $0<\alpha<1$ which is nonlinear of order $\beta\in (0,1)$:
\begin{equation}\label{Ex6}
^{C}D_{0+}^{\alpha}x(t)=f(x(t)), \quad x(0)=x_0,
\end{equation}
where
\begin{equation}\label{Ex6'}
f(x) :=  \begin{cases} -x^\beta, \quad\hbox{if}\quad x\geq 0,\\
|x|^\beta,\quad\hbox{if}\quad x<0.
\end{cases}
\end{equation}
It is worth mentioning that the function $f(\cdot)$ in right-hand side of the above FDE is continuous but non Lipschitzian in a neighborhood of the origin.

Let $x_0>0$, consider the FDE \eqref{Ex6} in the area $x\in (0,\infty)$. From Theorem \ref{exist_result_2}, the equation \eqref{Ex6} has a unique solution, denoted by $\varphi(\cdot,x_0)$, on the maximal interval of existence $[0,T_b)$. If $T_b(x_0)< \infty$, then $\liminf_{t\to T_b(x_0)-}\varphi(t,x_0)=0$ or $\limsup_{t\to T_b(x_0)-}\varphi(t,x_0)=\infty$ (see \cite[Proposition 1]{Feng-18a}). However, using Proposition \ref{comparision} and construction of a super-solution and a sub-solution to \eqref{Ex6} (see \cite[pp. 232--234]{zacher}), we have
\[
\limsup_{t\to T_b-}\varphi(t,x_0)\leq \frac{c_1}{1+T_b^{\alpha/\beta}}
\]
and
\[
\liminf_{t\to T_b-}\varphi(t,x_0)\geq \frac{c_2}{1+T_b^{\alpha/\beta}}
\]
for some $c_1,c_2>0$, a contradiction. Hence, $T_b=\infty$ and
\[
\frac{c_2}{1+t^{\alpha/\beta}}\leq \varphi(t,x_0)\leq \frac{c_1}{1+t^{\alpha/\beta}},\quad \forall t\geq 0.
\]
On the other hand, due to the specific form of $f$ in  \eqref{Ex6'}, if we multiply the solutions of \eqref{Ex6} with negative initial values by $-1$ then we get solutions of \eqref{Ex6} with the positive initial values,  and vice versa. Therefore, the solution of \eqref{Ex6} starting from $x_0\ne 0$ has decay rate as
$t^{-\gamma}$ with $\gamma = \alpha/\beta>\alpha$. This is different from the Lipschitz case (see Theorem~\ref{thm.power-rate}).

On the other hand, by a direct computation, we obtain a global solution of the initial value problem 
\[
\begin{cases}
^{C}D^\alpha_{0+}x(t)&=(x(t))^{\beta},\quad t>0.\\
x(0)&=0,
\end{cases}
\]
as $\varphi(t,0)=\left(\frac{\Gamma(1-\alpha)}{\frac{\alpha}{1-\beta}B(1-\alpha,\frac{\alpha}{1-\beta})}\right)^{1/(1-\beta)}t^{\alpha/(1-\beta)}$, where $\Gamma(\cdot)$ is Gamma function and $B(\cdot,\cdot)$ is Beta function. This implies that the trivial solution to \eqref{Ex6} is unstable.

A consequence of the  non-Lipschitz property at the origin of $f(\cdot)$ in this example is non-uniqueness of the solution: we have at least two solutions starting from the origin. This circumstance alone makes the system unstable although any solution starting from a point close to the origin but distinct from the origin tends to the origin with decay rate of  $t^{-\gamma}$.
\end{example}

Now we show that 
the Mittag-Leffler stability is strictly stronger than asymptotic stability. For this, we give below an example of  an asymptotically stable FDE  which is not Mittag-Leffler stable.
\begin{example}[Asymptotically stable nonlinear one-dimensional FDE which is not Mittag-Leffler stable]
 Consider a nonlinear one-dimensional FDE of order $0<\alpha<1$:
\begin{equation}\label{Ex3}
^{C}D_{0+}^{\alpha}x(t)=f(x(t)), \quad x(0)=x_0,
\end{equation}
where
\begin{equation}\label{Ex3'}
f(x) :=  \begin{cases} 
-e^{-1/x}x, &\quad\hbox{if}\quad x>0,\\
0, &\quad\hbox{if}\quad x=0,\\
-e^{1/x}x, &\quad\hbox{if}\quad x < 0.
\end{cases}
\end{equation}
Clearly $f(\cdot) \in C^2(-\infty,\infty)$. Therefore, by  \cite[Theorem 2]{Baleanu}
 the equation \eqref{Ex3} has a unique solution $x(\cdot)$ which exists globally on $\R_{\geq 0}$. 
 
 Fix some $x_0>0$. By \cite[Theorem 3.5]{Cong_Tuan_2016}, the solution of the FDE  \eqref{Ex3} cannot intersect the trivial solution, hence $x(t) >0$ for all $t\in\R_{\geq 0}$. 
 
 Now let $n\geq 2$ be an arbitrary integer.  Put $g(x):= -(n-1)!x^n$ on a neighborhood of 0 and extend it suitably to get $g(x) \leq f(x)$ on $(0,\infty)$. By Proposition \ref{comparision}, the solution $x(\cdot)$ of \eqref{Ex3} is bounded by the solution of the FDE 
 \begin{equation}\label{Ex3''}
 ^{C}D_{0+}^{\alpha}y(t)=g(y(t)), \quad y(0)=x_0.
 \end{equation}
 Using construction of a sub-solution by Vergara and Zacher \cite[pp. 332--334]{zacher}, we see that the solution $y(\cdot)$ of the FDE \eqref{Ex3''} has decay rate of $t^{-\alpha/n}$, hence the function $x(\cdot)$, which is bigger or equal to $y(\cdot)$, cannot converge faster than $t^{-\alpha/n}$. Since $n$ is arbitrary, $x(\cdot)$ cannot decay with power-rate. Thus, the trivial solution of \eqref{Ex3} is not Mittag-Leffler stable.
 
On the other hand, due to the fact that $f_{|(0,\infty)}\in C^2(0,\infty)$, using  \cite[Theorem 3.3]{Feng-18a}, we see that the solution $x(\cdot)$ of \eqref{Ex3} is  strictly decreasing on the interval $[0,\infty)$. Now we assume that there exist $\delta\in (0,1)$ such that $x(t)\geq \delta$ for all $t\geq 0$. Then,
\[
^{C}D^\alpha_{0+}x(t)\leq -e^{-1/\delta}x(t),\quad t>0.
\]
Using Proposition \ref{comparision}, we obtain
\[
x(t)\leq x_0 E_\alpha (-e^{-1/\delta}t)\to 0 \quad \textit{as}\quad t\to\infty,
\]
and we arrive at a contradiction. Consequently,   $x(\cdot)$ converges to $0$ as $t$ tends to $\infty$.  It is easily seen that 
this assertion is also true for the solution of \eqref{Ex3} starting from any $x_0<0$. 

Finally, since $f(\cdot) \in C^2(-\infty,\infty)$ the equation \eqref{Ex3} with the initial condition $x_0=0$ has the unique solution $x(t)\equiv 0$. Hence, the trivial solution of \eqref{Ex3} is asymptotically stable.  
\end{example}

To complete this section, we study the separation of trajectories of solutions to an one-dimensional FDE with local Lipschitz right-hand side defining on an interval $x\in(a,b)\subset \R$. We extend our previous result \cite[Theorem 3.5]{Cong_Tuan_2016} on separation of solution of one-dimensional FDE to this case. Let $-\infty \leq a<b\leq \infty$ and $f:[0,\infty)\times (a,b)\rightarrow \R$ be a continuous function and locally Lipschitz continuous with respect the second variable, that is, for any $T>0$ and any compact interval $K\subset (a,b)$ there exists a positive constant $L_{K,T}$ such that
\begin{equation}\label{Lip_cond}
|f(t,x)-f(t,y)|\leq L_{K,T}|x-y|,\quad \forall x,y\in K,\;t\in [0,T].
\end{equation}
Consider the equation 
\begin{equation}\label{l_eq}
^{C}D^\alpha_{0+}x(t)=f(t,x(t)),\quad t>0.
\end{equation}
Then, using the approach of \cite{Cong_Tuan_2016} we obtain the following result.
\begin{theorem}\label{separation}
Assume that the function $f(\cdot,\cdot)$ satisfies the condition \eqref{Lip_cond}. Then, for any pair of distinct points  $x_1, x_2\in (a,b)$ the solutions of the FDE  \eqref{l_eq} starting from $x_1$ and $x_2$, respectively, do not meet.
\end{theorem}
\begin{proof}
By virtue Theorem \ref{exist_result_2}, for $x_i\in (a,b)$ the initial value problem \eqref{l_eq}, $x(0)=x_i$ ($i=1,2$), has the unique solution denoted by $\varphi(\cdot,x_i)$ on the maximal interval of existence $[0,T_b(x_i))$. Without loss of generality we let $x_1<x_2$. Assume that $\varphi(\cdot,x_1)$ and $\varphi(\cdot,x_2)$ meet at some $t\in (0,T_b(x_1))\cap (0,T_b(x_2))$. Let $t_1:=\inf\{t\in (0,T_b(x_1))\cap (0,T_b(x_2)): \varphi(t,x_1)=\varphi(t,x_2)\}$. It is obvious that $0<t_1<\min\{T_b(x_1), T_b(x_1)\}$ and
\[
\varphi(t_1,x_1)=\varphi(t_1,x_2),\quad \varphi(t,x_1)<\varphi(t,x_2),\quad \forall t\in [0,t_1).
\] 
Take $r_1,r_2>0$ such that $[x_1-r_1,x_2+r_2]\subset (a,b)$ and $\varphi(t,x_1),\varphi(t,x_2)\in [x_1-r_1,x_2+r_2]$ for all $t\in [0,t_1]$. Then following the assumption on the locally Lipschitz continuity of $f(\cdot,\cdot)$ (see the condition \eqref{Lip_cond}), the function
\[
f_1:=f_{|[0,t_1]\times [x_1-r_1,x_2+r_2]}
\]
is continuous and Lipschitz continuous with respect to the second variable on the set $[0,t_1]\times [x_1-r_1,x_2+r_2]$. 

Now we construct a extension of $f_1(\cdot,\cdot)$ as follows:
\[
f_2(t,x):=\begin{cases}
f_1(t,x),\quad \text{if}\quad (t,x)\in [0,t_1]\times [x_1-r_1,x_2+r_2],\\
f_1(t,x_2+r_2),\quad \text{if}\quad t\in [0,t_1],\;x>x_2+r_2,\\
f_1(t,x_1-r_1),\quad \text{if}\quad t\in [0,t_1],\;x<x_1-r_1.
\end{cases}
\]
This function is continuous and global Lipschitz continuous with respect to the second variable on the domain $[0,t_1]\times \R$. Therefore, by \cite[Theorem 2]{Baleanu} the FDE
\begin{equation}\label{l_1_eq}
^{C}D^\alpha_{0+}x(t)=f_2(t,x(t)),\quad t>0,\quad x(0)=x_i,\; i=1,2,
\end{equation}
has  unique solutions $\tilde{\varphi}(\cdot,x_i)$, $i=1,2$, on $\R_{\geq 0}$. On the other hand, using \cite[Theorem 3.5]{Cong_Tuan_2016}, we have
\[
\tilde{\varphi}(t,x_1)<\tilde{\varphi}(t,x_2),\quad \forall t\in \R_{\geq 0}.
\]
However, due to the fact $\varphi(t,x_1),\varphi(t,x_2)\in [x_1-r_1,x_2+r_2]$ for all $t\in [0,t_1]$, we also have
\begin{align*}
^{C}D^\alpha_{0+}\varphi(t,x_i)&=f(t,\varphi(t,x_i))\\
&=f_2(t,\varphi(t,x_i)),\quad t\in (0,t_1],\;i=1,2.
\end{align*}
This implies that
\[
\varphi(t,x_1)=\tilde{\varphi}(t,x_1)<\tilde{\varphi}(t,x_2)=\varphi(t,x_2)
\]
for all $t\in [0,t_1]$, a contradiction. Thus two solutions $\varphi(\cdot,x_1)$ and $\varphi(\cdot,x_2)$ do not meet and the proof is complete.
\end{proof}
\begin{remark}
 (i) Theorem \ref{separation} improves our preceding result \cite[Theorem 3.5]{Cong_Tuan_2016}. Here, we only used the assumption on the locally Lipschitz continuity of "vector field" $f(\cdot,\cdot)$ instead of the global Lipschitz continuity of this function.

(ii) This theorem also improved a recent result by Y. Feng \textit{et al.} \cite[proposition 2]{Feng-18a}. More precisely, we removed the condition on monotonity of the function $f(\cdot)$ in \cite[Proposition 2]{Feng-18a} (see also \cite[Remark 6]{Feng-18a}).
\end{remark}
\section*{Acknowledgments}
This research is funded by Vietnam National Foundation for Science and Technology Development (NAFOSTED) under grant number FWO.101.2017.01. The authors thank Prof. Doan Thai Son for helpful discussions.


\begin{thebibliography}{15}
%
\bibitem{A_17}
M. Aghababa. Stabilization of a class of fractional-order chaotic systems using a non-smooth control methodology. {\em Nonlinear Dynamics}, {\bf 89} (2017), Issue 2, 1357--
1370.
%
%
\bibitem{CMC_14}
N. Aguila-Camacho, M. Duarte-Mermoud, J. Gallegos. Lyapunov functions for fractional order systems. {\em Commun Nonlinear Sci Numer Simulat.}, {\bf 19} (2014), Issue 9, 2951--2957.
%
\bibitem{A_07}
E. Ahmed, A. El-Sayed, H. El-Saka. Equilibrium points, stability and numerical solutions of fractional-order predator-prey and
rabies models. {\em J. Math. Anal. Appl.}, {\bf 325} (2007), 542--553.
%
%
\bibitem{Baleanu}
D.~Baleanu, O.~Mustafa.
\newblock{On the global existence of solutions to a class of fractional differential equations.}
\newblock{\em Computers and Mathematics with Applications,} {\bf 59} (2010), 1583--1841.
%
\bibitem{BK_15}
B. Bandyopadhyay, S. Kamal. {\em Stabilization and Control of Fractional Order Systems: A Sliding Mode Approach}. Lecture Notes in Electrical Engineering {\bf 317}.
Springer International Publishing, Switzerland, 2015.
%

%
\bibitem{CCWY_12}
L. Chen, Y. Chai, R. Wu, J. Yang. Stability and stabilization of a class of nonlinear fractional order system with Caputo derivative. {\em IEEE Transactions on Circuits and Systems-II: Express
Briefs}, {\bf 59} (2012), No. 9, 602--606.
%
\bibitem{CDSZ_17}
W. Chen, H. Dai, Y. Song, Z. Zhang. Convex Lyapunov functions for stability analysis of fractional order systems. {\em IET Control Theory and Applications}, {\bf 11} (2017), Issue 7,
1070--1074.
%
\bibitem{Choi-14}
S. Choi, B. Kang, N. Koo.
 Stability for Caputo Fractional Differential Systems.
 {\em Abstract and Applied Analysis}, 
 Volume 2014, Article ID 631419, 6 pages. 
http://dx.doi.org/10.1155/2014/631419
%
\bibitem{CST_14}
N.D. Cong, T.S. Doan, H.T. Tuan. On fractional lyapunov exponent for solutions of linear fractional differential equations. {\em Fractional Calculus and Applied Analysis}, {\bf 17} (2014), 285--306.
%
\bibitem{CST_16}
N.D. Cong, T.S. Doan, S. Stefan, H.T. Tuan. Linearized asymptotic stability for fractional differential equations. {\em Electronic Journal of Qualitative Theory of Differential Equations}, {\bf 39} (2016), 1--13.
%
\bibitem{Cong_Tuan_2016}
N.D. Cong, H.T. Tuan.
\newblock{Generation of nonlocal fractional dynamical systems by fractional differential equations.} {\em Journal of Integral Equations and Applications}, {\bf 29} (2017), 1--24.
%
\bibitem{Cong_Tuan_2018}
N.D. Cong, T.S. Doan, H.T. Tuan. Asymptotic stability of linear fractional systems with constant coefficients and small time dependent perturbations. {\em Vietnam Journal of Mathematics}, {\bf 46} (2018), 665--680.
%
%
\bibitem{Kai}
K.~Diethelm.
\newblock{\em The Analysis of Fractional Differential Equations. An Application-oriented Exposition Using Differential Operators of Caputo Type.}
\newblock{\em  Lecture Notes in Mathematics} {\bf 2004}.
\newblock{Springer-Verlag, Berlin, 2010.}
%
\bibitem{DQW_14}
D. Ding, D. Qi, Q. Wang. Nonlinear Mittag-Leffler stabilisation of commensurate fractional order nonlinear systems. {\em IET Control Theory and Applications}, {\bf 9} (2014), Issue 5, 681--690.
%
%
\bibitem{MCG_15}
M. Duarte-Mermoud, N. Aguila-Camacho, J. Gallegos. Using general
quadratic Lyapunov functions to prove Lyapunov uniform stability for fractional
order systems. {\em Commun Nonlinear Sci Numer Simulat.}, {\bf 22} (2015), Issues 1--3, 650--659.
%
\bibitem{Feng-18}
Y. Feng, L. Li, J.-G. Liu, X. Xu.
A note on one-dimensional time fractional ODEs.
{\em Applied Mathematics Letters,} {\bf 83} (2018), 87--94.
%
\bibitem{Feng-18a}
Y. Feng, L. Li, J.-G. Liu, X. Xu.
Continuous and discrete one dimensional autonomous fractional ODEs.
{\em Discrete and continuous dynamical systems--Series B,} {\bf 23} (2018), No. 8, 3109--3135. doi:10.3934/dcdsb.2017210.
%

%
\bibitem{KST_06}
A. Kilbas, H. Srivastava, J. Trujillo. {\em Theory and Applications of Fractional Differential Equations}. North-Holland Mathematics
Studies {\bf 204}. Elsevier Science B.V., Amsterdam, 2006.
%
\bibitem{LLD_09}
V. Lakshmikantham, S. Leela, J. Devi. {\em Theory of Fractional Dynamic Systems}. Cambridge Scientific Publishers Ltd., England, 2009.
%
\bibitem{LeiLi}
L. Li, J. Liu. A generalized definition of Caputo derivatives and its applications to fractional ODEs. {\em SIAM J. Math. Anal.,} {\bf 50} (2018), No. 3, 2867--2900.
%
\bibitem{Li_1}
Y. Li, Y. Chen, I. Podlubny. Mittag-Leffler stability of fractional order nonlinear dynamic
systems. {\em Automatica}, {\bf 45} (2009), 1965--1969. 
%
\bibitem{Li_2}
Y. Li, Y. Chen, I. Podlubny. Stability of fractional-order nonlinear dynamic system: Lyapunov direct method and generalized Mittag-Leffler stability. {\em Comput. Math. Appl.,}
{\bf 59} (2010), 1810--1821.

%
\bibitem{OS_74}
K. Oldham, J. Spanier. {\em The Fractional Calculus}. Academic Press, New York, 1974.
%
\bibitem{Podlubny}
I.~Podlubny.
\newblock{\em Fractional Differential equations. An Introduction to Fractional Derivatives, Fractional Differential Equations, to Methods of their Solution and some of Their Applications.}
\newblock{ Mathematics in Science and Engineering \textbf{198}.}
\newblock{ Academic Press, Inc., CA, 1999.}
%
\bibitem{vastala}
J.D. Ramirez, A.S. Vatsala. Generalized monotone iterative technique for Caputo fractional differential equation with
periodic boundary condition via initial value problem. {\em Int. J. Differ. Equ.}, {\bf 2012} (2012), 1--17.
%
\bibitem{SKM}
S. Samko, A. Kilbas, O. Marichev. {\em Fractional Integrals and Derivatives:
Theory and Applications}. Gordon and Breach Science Publishers, USA, 1993.
%
\bibitem{Shen}
J. Shen, J. Lam. Non-existence of finite-time stable equilibria in fractional-order nonlinear systems. \textit{Automatica}, {\bf 50} (2014), 547--551.
%
\bibitem{Shilov}
G.~Shilov.
\newblock{\em Linear Algebra.}
\newblock{ Dover Publications, Inc., New York, 1977}.
%
\bibitem{Stamova_15}
I. Stamova. Mittag–Leffler stability of impulsive differential equations of fractional order. {\em Q. Appl. Math.,} {\bf 73} (2015), No. 3, 239--244.
%
\bibitem{Tuan_Hieu}
H.T. Tuan, Hieu Trinh. Stability of fractional-order nonlinear systems by Lyapunov direct method. To appear in {\em IET Control Theory and Applications}. Doi: 10.1049/iet-cta.2018.5233.
%
\bibitem{Vainikko}
G. Vainikko. Which functions are fractionally differentiable? \textit{Journal of Analysis and its Applications,} {\bf 35} (2016), 465--487.
%
\bibitem{zacher}
V. Vergara, R. Zacher. Optimal decay estimates for time-fractional and other nonlocal subdiffusion equations via energy
methods. {\em SIAM J. Math. Anal.,} {\bf 47} (2015), No. 1, 210--239.
%
\bibitem{WWL}
X. Wen, Z. Wu, J. Lu. Stability analysis of a class of nonlinear fractional-order
systems. {\em IEEE Trans. Circuits Syst-II: Express Briefs}, {\bf 55} (2008), No. 11, 1178--1182.
%
\bibitem{YC_16}
Y. Yunquan, M. Chunfang. Mittag-Leffler stability of fractional order Lorenz and Lorenz family systems. {\em Nonlinear Dynamics}, {\bf 83} (2016), Issue 3, 1237--1246.
%
\bibitem{ZTC_15}
R. Zhang, G. Tian, S. Yang, H. Cao. Stability analysis of a class of fractional order nonlinear systems with order lying in (0,2). {\em ISA Trans.}, {\bf 56} (2015), 102--110. 
%
\bibitem{Zhou}
X. Zhou, L. Hu, W. Jiang. Stability criterion for a class of nonlinear fractional differential systems. \textit{Applied Mathematics Letters}, {\bf 28} (2014), 25--29. 
%
\end{thebibliography}
\end{document}